\documentclass[11pt]{amsart}

\usepackage[utf8]{inputenc} 
\usepackage[a4paper]{geometry} 
\usepackage{enumerate}  
\usepackage{amsmath} 
\usepackage{amssymb}
\usepackage{float}
\usepackage{color}
\usepackage{multirow}
\usepackage{hyperref}  
\usepackage{verbatim}
                 \hypersetup{ pdfborder={0 0 0}, 
                              colorlinks=true,
                              linktoc=page,
                              pdfauthor={Mauro Fortuna, Michael Hoff and Giacomo Mezzedimi}, 
                              pdftitle={Unirational moduli spaces of elliptic K3 surfaces}} 
\definecolor{Gray}{gray}{0.9}                            
\usepackage[arrow,curve,matrix,tips,frame]{xy}

\theoremstyle{plain} 
\newtheorem{proposition}{Proposition}[section] 
\newtheorem{theorem}[proposition]{Theorem} 
\newtheorem{lemma}[proposition]{Lemma}

\theoremstyle{definition}

\newtheorem{example}[proposition]{Example} 
\theoremstyle{remark} 
\newtheorem{remark}[proposition]{Remark}

\renewcommand{\O}{{\mathcal{O}}}   
\newcommand{\cO}{{\mathcal{O}}}

\newcommand{\CC}{{\mathbb{C}}}      
\newcommand{\PP}{{\mathbb{P}}}

\newcommand{\cM}{{\mathcal{M}}}

\providecommand{\rk}{\mathop{\rm rk}}

\numberwithin{equation}{section}

\title[Unirational moduli spaces of some elliptic K3 surfaces]{Unirational moduli spaces of some elliptic K3 surfaces}

\author[M. Fortuna]{Mauro Fortuna}
\address{Institut für Algebraische Geometrie, Leibniz Universität Hannover, Welfengarten 1, 30167 Hannover, Germany.}
\email{\href{fortuna@math.uni-hannover.de}{fortuna@math.uni-hannover.de}}

\author[M. Hoff]{Michael Hoff} 
\address{Universit\"at des Saarlandes, Campus E2 4, D-66123 Saarbr\"ucken, Germany}
\email{\href{mailto:hahn@math.uni-sb.de}{hahn@math.uni-sb.de}} 

\author[G. Mezzedimi]{Giacomo Mezzedimi}
\address{Institut für Algebraische Geometrie, Leibniz Universität Hannover, Welfengarten 1, 30167 Hannover, Germany.}
\email{\href{mezzedimi@math.uni-hannover.de}{mezzedimi@math.uni-hannover.de}}

\address{\textit{Current address:} Mathematisches Institut \\ Universität Bonn \\ Endenicher Allee 60 \\ 53115 Bonn \\ Germany}
\email{\href{mezzedim@math.uni-bonn.de}{mezzedim@math.uni-bonn.de}}

\date{\today} 


\begin{document}
\begin{abstract} We show that the moduli space of $U\oplus \langle -2k \rangle$-polarized K3 surfaces is unirational for $k \le 50$ and $k \notin \{11,35,42,48\}$, and for other several values of $k$ up to $k=97$. Our proof is based on a systematic study of the projective models of elliptic K3 surfaces in $\PP^n$ for $3\le n \le 5$ containing either the union of two smooth rational curves or the union of a smooth rational curve and an elliptic curve intersecting at one point.
\end{abstract}

\maketitle

\section{Introduction}

By classical results and works of Mukai \cite{Muk88, Mukg11, Mukg13, Mukg16, Mukg1820}, it is known that the moduli spaces of complex K3 surfaces of genus $g \le 12$ and $g=13, 16, 18, 20$ are unirational. This was later improved by Farkas and Verra \cite{FV18, FV19} extending the unirationality result to K3 surfaces of genus $g=14, 22$ by using the connection to special cubic fourfolds. Recently, the moduli spaces of $n$-pointed K3 surfaces of genus $g\le 22$ were studied systematically in \cite{Ma19}. It is then natural to ask the more general question about the unirationality of moduli spaces of lattice polarized K3 surfaces. Farkas and Verra \cite{FV12,FV16, V15} worked out the case of polarized Nikulin surfaces. The case of $2$-elementary K3 surfaces was studied in \cite{Ma15}, and further results in this direction were obtained in \cite{BHK16} by using orbits of representation of algebraic groups. In the present article, we will restrict to the case of elliptic K3 surfaces of Picard rank at least $3$.

We recall that a K3 surface $S$ is called elliptic if it admits a fibration $S\to \PP^1$ in curves of genus one together with a section. The geometry of elliptic surfaces can be studied via their realization as Weierstrass fibrations. By using this description in \cite{Mir81}, Miranda constructed the moduli space of elliptic K3 surfaces and showed its unirationality (it is actually rational by \cite{Lej93}). The N\'eron-Severi group of the very general elliptic K3 surface is isomorphic to the hyperbolic plane $U$, and it is generated by the classes of the fiber and the zero section of the elliptic fibration. 

For an elliptic K3 surface of Picard rank $3$, its Picard lattice is isomorphic to $U\oplus \langle -2k \rangle$ for some integer $k\ge 1$. We are interested in studying the moduli spaces $\mathcal{M}_{2k}$ of $U\oplus \langle -2k \rangle$-polarized K3 surfaces. These moduli spaces are divisors in the moduli space of elliptic K3 surfaces. 

The study of $\mathcal{M}_{2k}$ was initiated in \cite{FM20}, where the authors showed the unirationality of $\mathcal{M}_{2k}$ for some values of $k\le 64$. Moreover, in the same article it is proven that the Kodaira dimension of $\cM_{2k}$ is non-negative for $k \ge 176$ (and for other smaller values until $k=140$). We extend the unirationality result using the computer algebra system \emph{Macaulay2}. 
We construct projective models of $U\oplus \langle -2k \rangle$-polarized K3 surfaces in $\PP^n$ for $3 \le n \le 5$ for several values of $k$, leading to the following theorem.

\begin{theorem} \label{thm:principal}
The moduli space $\mathcal{M}_{2k}$ is unirational for the following values of $k$: 
\begin{align*}
\{ &  {6}, {7}, {9},   {10}, 12,\dots, 34, 36, \dots, 41, 43, \dots, 47, {49},   {50},\\
 &       52,    {53}, 54,  {59}, 60, 61, 62,  {64},  68,   {69},   {73},   {79},   {81},   {94},   {97} \}.
\end{align*}
\end{theorem}

Combining this with the previous work in \cite{FM20}, that constructed the moduli spaces $\cM_{2k}$ using double covers and Weierstrass fibrations, we obtain a more complete result:

\begin{theorem} 
The moduli space $\mathcal{M}_{2k}$ is unirational for $k\le 50$, $k \notin \{11,35,42,48\}$ and for the following values of $k$: 
\begin{align*}
\{52,    {53}, 54,  {59}, 60, 61, 62,  {64},  68,   {69},   {73},   {79},   {81},   {94},   {97} \}.
\end{align*}
\end{theorem}

Our strategy involves a systematic study of the projective models of $U\oplus \langle -2k \rangle$-polarized K3 surfaces in $\PP^3$, $\PP^4$ and $\PP^5$ containing either two smooth rational curves or an elliptic curve and a smooth rational curve meeting at one point. We observe that similar ideas could be used to find new projective models, allowing curves of higher genus or working in higher dimensional projective spaces. The first approach seems however to produce no new cases. On the other hand, our method leads to the unirationality of the moduli spaces of many other lattice polarized (non-elliptic) K3 surfaces of Picard rank at least $3$; we have not included these ones in this article.

This research leads to several follow-up questions. First of all, one would like to understand the Kodaira dimension of $\cM_{2k}$ for the values of $k$ in the remaining gaps. Moreover, it would be of great interest to find connections between the spaces $\cM_{2k}$ and other known geometric objects. One such connection was found in \cite{BH17}, where the authors showed that the moduli space $\cM_{56}$ is birational to a $\PP^1$-bundle over the universal Brill--Noether variety $\mathcal{W}_{9,6}^1$ parametrizing curves of genus $9$ together with a pencil of degree $6$. Their strategy was to use the relative canonical resolution of these curves on rational normal quartic scrolls. We hope that our methods can be used to reveal new reincarnations of the moduli spaces $\cM_{2k}$ since we provide explicit methods to study such elliptic K3 surfaces and the geometry of their moduli spaces.

The article is organized as follows. In Section \ref{sec:generalstrategy} we explain the general strategy to prove the main theorem. More precisely, we first find an exhaustive list of possible projective models for $U \oplus \langle -2k \rangle$-polarized K3 surfaces as complete intersection in $\PP^3$, $\PP^4$ and $\PP^5$. This is then used to construct a dominant rational map $I \dashrightarrow \mathcal{M}_{2k}$ (as in diagram \ref{incidenceVarietyK3}), where $I$ is a unirational parameter space (cf. Section \ref{unir}). Section \ref{experimentalData} contains the experimental data obtained with our program in \emph{Macaulay2}, with some examples. Finally, we extend  these contructions in Section \ref{nodalK3} to the case of nodal elliptic K3 surfaces.

All our constructions are implemented in \emph{Macaulay2} (see \cite{macaulay2}) and are available at the authors' homepage \cite{M2support}.

\subsection*{Acknowledgements} We would like to thank Klaus Hulek and Matthias Schütt for useful discussions and for reading an early draft of this manuscript. We also thank the anonymous referee for carefully reading the paper and suggesting several improvements. The first author acknowledges partial support from the DFG Grant Hu 337/7-1.

\section{General strategy} \label{sec:generalstrategy}

Our strategy is to construct projective models of K3 surfaces in $\PP^n$ for $3 \le n \le 5$ containing suitable pairs of curves. Let $U$ be the hyperbolic plane. We focus on elliptic K3 surfaces of Picard rank at least $3$, that is, $U\oplus \langle -2k \rangle$-polarized K3 surfaces for some $k \ge 1$.
We first recall the construction of the moduli space of lattice polarized K3 surfaces. Our main reference is \cite{Do96}.

For a hyperbolic lattice $L$ embedding primitively in the K3 lattice $\Lambda_{K3}=U^3\oplus E_8^2$, the moduli space $\mathcal{M}_L$ of $L$-polarized K3 surface is constructed as follows. Denote by $T=L^\perp_{\Lambda_{K3}}$ the orthogonal complement of $L$ in $\Lambda_{K3}$.
Let $\Omega_T$ be one of the two connected components of $\{w\in \PP(T\otimes \CC) : (w,w)=0, (w,\overline{w})>0\}$, and consider the group $\mathrm{O}^+(T)$ of isometries of $T$ preserving $\Omega_T$. We denote by $\widetilde{\mathrm{O}}^+(T)$ the subgroup of $\mathrm{O}^+(T)$ of isometries acting trivially on the discriminant group $A_T=T^\vee/T$ of $T$, where $T^\vee$ is the dual lattice of $T$ (see \cite[Section~1.3]{Nik79} for the precise lattice-theoretical definitions). Then $\mathcal{M}_L$ is defined as the quotient $\widetilde{\mathrm{O}}^+(T)\backslash \Omega_T$. By the classical result of Baily and Borel \cite{BB66}, $\mathcal{M}_L$ is a quasi-projective variety of dimension $\rk(T)-2=20-\rk(L)$.

An \emph{$L$-polarized K3 surface} is a pair $(X,j)$, where $X$ is a K3 surface and $j:L\hookrightarrow \mathrm{NS}(X)$ is a primitive embedding of $L$ in the N\'eron-Severi group of $X$. The pair $(X,j)$ defines the point in $\mathcal{M}_L$ corresponding to the unique (up to scalars) nondegenerate holomorphic $2$-form $\omega_X\in j(L)^\perp_{\Lambda_{K3}}$ on $X$. By the Torelli theorem, $\mathcal{M}_L$ is the coarse moduli space of $L$-polarized K3 surfaces (cf. \cite[Corollary 3.2]{Do96}).

In our case we have $L_{2k}=U\oplus \langle -2k\rangle$, so the moduli spaces $\mathcal{M}_{2k}:=\mathcal{M}_{L_{2k}}$ we are interested in have dimension $17$.

For any fixed $k$, a very general $U\oplus \langle -2k \rangle$-polarized K3 surface $X$ has $\mathrm{NS}(X)=U\oplus \langle -2k \rangle$. In order to prove unirationality of some of these moduli spaces, we ask $\mathrm{NS}(X)$ to admit a basis given by a very ample polarization and two other smooth curves, 
where we distinct  two cases: 
\begin{itemize}
    \item[Case 1:] Two smooth rational curves meeting transversely at several points in general position;
    \item[Case 2:] A smooth rational curve and an elliptic curve meeting transversely at one point.
\end{itemize}

The reason for this assumption is that we want to realize these K3 surfaces as complete intersections in $\PP^n$, $3\le n\le 5$, and therefore we need a very ample divisor on $X$ of square $2n-2$ providing such a projective model. Since our K3 surfaces have Picard rank $3$, we need two extra curve classes, and it turns out that allowing for the two cases listed above gives already a lot of flexibility, enabling us to show the unirationality of many of the moduli spaces $\mathcal{M}_{2k}$. Experimental data suggests that allowing two extra classes of curves of higher genus does not help in proving the unirationality of the remaining spaces $\mathcal{M}_{2k}$, so one probably needs to consider K3 surfaces with polarizations of higher degree in order to extend the results of this paper.

Notice that in Case 2 the resulting K3 surfaces are automatically elliptic, since the elliptic curve induces the desired elliptic fibration, and the smooth rational curve becomes a section of the fibration.
On the other hand, the K3 surfaces resulting from Case 1 are elliptic only in some cases, depending on the arithmetic of the N\'eron-Severi lattice resulting from the construction. 
We explain in Section \ref{sub:search} how to check whether the constructed K3 surfaces are actually elliptic.

\subsection{The construction}\label{construction}

We describe the construction in detail for Case 2. Let $n,d,\gamma$ be integers such that $3\le n \le 5$, $d \ge 3$ and  $\gamma \ge 1$. 
\begin{itemize}
 \item [Step 1:] We construct a smooth elliptic curve $E$ of degree $d$ in $\PP^n$ with a distinguished point $p$.
 \item [Step 2:] We construct a smooth rational curve $\Gamma$ of degree $\gamma$ intersecting $E$ transversely only at the point $p$.
 \item [Step 3:] We choose (if it exists) a smooth K3 surface $X$ in $\PP^n$ of degree $2n-2$ containing $E\cup \Gamma$. 
\end{itemize}
Then, we get a K3 surface $X$ containing an elliptic curve $E$ and a smooth rational curve $\Gamma$ with the following lattice embedding
\begin{equation}\label{genericLattice}
\begin{pmatrix}
 2n-2 & \gamma & d \\
 \gamma & -2 & 1 \\ 
 d & 1 & 0
\end{pmatrix}
\cong U \oplus \langle 2n-2 - 2d\gamma - 2d^2 \rangle  \hookrightarrow \mathrm{NS}(X).
\end{equation}
In particular, we set $k = d^2 +d\gamma -n+1$. We check case by case that $X$ is a $U\oplus \langle -2k \rangle$-polarized K3 surface by showing that such a lattice embedding is always primitive (see Section \ref{prim}). By abuse of notation, we denote by $E$ and $\Gamma$ the classes in $\mathrm{NS}(X)$ of the corresponding curves under this lattice embedding. Let $\cM_{2k}$ be the moduli space of $U\oplus \langle -2k \rangle$-polarized K3 surfaces. $\cM_{2k}$ is an irreducible variety of dimension $17$.

We can easily adapt the above strategy to Case 1. First, we construct a smooth rational curve $\Gamma_1 \subseteq \PP^n$ of degree $\gamma_1$, together with $m$ points $p_1,\ldots,p_m\in \Gamma_1$. Then, we construct a second smooth rational curve $\Gamma_2$ intersecting $\Gamma_1$ transversely, precisely at $p_1,\ldots,p_m$. Finally, Step 3 remains unchanged: we just choose (if it exists) a smooth K3 surface $X$ in $\PP^n$ containing $\Gamma_1 \cup \Gamma_2$.\newline

We compute in \textit{Macaulay2} that the constructed curves are smooth points of a component of the right dimension in the corresponding Hilbert schemes. By standard semicontinuity arguments (see e.g. \cite{SchHandbook}), we will perform our computations over a finite field (the main reason for doing this is that the computation is much faster over a finite field, but our constructions also work over the rationals)\color{black}. Finally, a dimension count shows that the construction dominates the corresponding moduli space $\cM_{2k}$. We will present more details in the rest of the section.

\subsection{Unirationality} \label{unir}

The constructions described in Step 1, 2 and 3 can be realized as incidence varieties which are shown to be unirational.

\begin{remark}
We denote by $H_{d,g,n}$ the open subscheme of the Hilbert scheme parametrizing smooth irreducible curves of degree $d$ and genus $g$ in $\PP^n$. We notice that $H_{d,g,n}$ is irreducible if $g\in \{0,1\}$ and $n\in \{3,4,5\}$ by \cite{Ein}. Moreover, we can easily compute the dimension of $H_{d,g,n}$ for $g\in \{0,1\}$ by using the fact that every smooth curve $C\subseteq \PP^n$ of genus $g \le 1$ and degree $d>0$ is \textit{non-special}, that is, $H^1(\mathcal{O}_C(1))=0$. Indeed, by using the Euler sequence and the defining sequence for the normal bundle of $C\subseteq \PP^n$, we deduce that $H^1(\mathcal{N}_{C/\PP^n})=0$, and thus
$$\dim{H_{d,g,n}}=h^0(\mathcal{N}_{C/\PP^n})=\begin{cases}
(n+1)d+(n-3) &\text{if } g=0\\
(n+1)d &\text{if } g=1.
\end{cases}$$
\end{remark}

\subsection*{Step 1}

We include the proofs of the following classical results for the sake of completeness. 

\begin{lemma}
 Let $n,d\ge 3$ be integers.
 The incidence variety
 $$
 H_{d,1,n}^1 := \{(E,p)\ |\ p\in E\} \subseteq H_{d,1,n} \times \PP^n
 $$
 of elliptic curves of degree $d$ in $\PP^n$ with a marked point is unirational. Its dimension is 
 $$
 \dim H_{d,1,n}^1 = \dim H_{d,1,n} + 1.
 $$
\end{lemma}
\begin{proof}
The moduli space $\cM_{1,2}$ of elliptic curves marked with $2$ points is rational by \cite[Lemmas~1.1.3~and~1.1.4]{Be98}. In order to construct an elliptic curve $E$ of degree $d\ge 3$ in $\PP^n$ together with a marked point $p$, we start with a plane cubic curve $E'$ with two distinguished points $p$ and $q$. The choice of a basis of the vector space $V = H^0(E',\O_{E'}(d q))$ of dimension $d$ yields a birational map
$$
E' \longrightarrow E\subset \PP^{d-1}
$$
where the image $E$ is a smooth elliptic curve of degree $d$ (recall that all line bundles on $E'$ of degree $d$ are of the form $\O_{E'}(d q)$, see e.g. \cite[Theorem 6.16]{eisenbud-sy}). The choice of the basis is unirational since it is parametrized by an open subset of $V^d$. If $d-1>n$, then we project the curve $E$ birationally to a smooth elliptic curve of degree $d$ in $\PP^n$. If $d-1<n$, then we embed the ambient space $\PP^{d-1}$ into $\PP^n$ in order to get again an elliptic curve of degree $d$ in $\PP^n$. In both cases, we have to choose an appropriate linear subspace of $\PP^n$ or $\PP^{d-1}$, and this choice is clearly unirational.

Now consider the forgetful morphism $\phi :H_{d,1,n}^1\rightarrow H_{d,1,n}$. $\phi$ is dominant, since every elliptic curve of degree $d$ in $\PP^{d-1}$ arises as the image of a (plane or abstract) elliptic curve via a complete linear system of degree $d$, and every elliptic curve of degree $d$ in $\PP^n$ lies in a linear subspace of dimension $d-1$ if $n>d-1$.
Since the preimage of a general point $[E]\in H_{d,1,n}$ is isomorphic to $E$ itself, we conclude that $\dim H_{d,1,n}^1 = \dim H_{d,1,n} + 1$.
\end{proof}

\begin{remark}
Let $p_1,\ldots,p_m \in \PP^n$ be a set of points spanning a linear subspace $\PP^l \subseteq \PP^n$. We say that $p_1,\ldots,p_m$ are in general position if they are in general position inside $\PP^l$. In particular, we have $m \le l+1 \le n+1$.  
\end{remark}

\begin{lemma} \label{lemma:2rat}
Let $n\ge 3$, $\gamma\ge 1$ and $m \le \min\{n+1,\gamma+1\}$ be integers.
Fix points $p_1,\ldots,p_m \in \PP^n$ in general position. The variety
 $$
 H_{\gamma,0,n}(p_1,\ldots,p_m) := \{\Gamma \ni p_1,\ldots,p_m\} \subseteq H_{\gamma,0,n}
 $$
 of smooth rational curves of degree $\gamma$ in $\PP^n$ passing through $p_1,\ldots,p_m$ is irreducible and unirational. Moreover, it is non-empty, of dimension 
 $$
 \dim(H_{\gamma,0,n}(p_1,\ldots,p_m)) = \dim(H_{\gamma,0,n})-m(n-1).
 $$
\end{lemma}
\begin{proof}
Since the variety of embeddings $\mathrm{Emb}_\gamma(\PP^1,\PP^n)$ of degree $\gamma$ is rational and the morphism $\mathrm{Emb}_\gamma(\PP^1,\PP^n)\rightarrow H_{\gamma,0,n}$ sending a morphism to its image is dominant, we have that $H_{\gamma,0,n}$ is irreducible and unirational. If $p_1,\ldots,p_m$ are points in $\PP^n$ in general position, we consider the subvariety $T$ of $\mathrm{Emb}_\gamma(\PP^1,\PP^n)\times (\PP^1)^m$ consisting of embeddings $f:\PP^1\hookrightarrow \PP^n$ and $m$ points $x_1,\ldots,x_m \in \PP^1$ such that $f(x_i)=p_i$ for all $1 \le i \le m$. Since the conditions $f(x_i)=p_i$ are linear in the coefficients of the polynomials defining the embedding $f$, and $T$ dominates $H_{\gamma,0,n}(p_1,\ldots,p_m)$, we deduce that $H_{\gamma,0,n}(p_1,\ldots,p_m)$ is irreducible and unirational.

In order to show that $H_{\gamma,0,n}(p_1,\ldots,p_m)$ is non-empty, fix a linear map $h:\PP^\gamma \rightarrow \PP^n$ and points $q_1,\ldots,q_m\in \PP^\gamma$ in general position such that $h(q_i)=p_i$ for all $1 \le i \le m$. Since there is always a smooth rational normal curve $C\subseteq \PP^\gamma$ of degree $\gamma$ passing through $m \le \gamma+1$ points in general position, then $\Gamma := h(C)\subseteq \PP^n$ is the desired smooth rational curve.

Finally, we have to compute $\dim(H_{\gamma,0,n}(p_1,\ldots,p_m))$. Let $\Gamma \in H_{\gamma,0,n}(p_1,\ldots,p_m)$ be a smooth rational curve as constructed above, and consider the divisor $D:=p_1+\ldots+p_m$ over $\Gamma$. Let $H$ be the restriction of the hyperplane class on $\PP^n$ to $\Gamma$. Then the exact sequences
$$0 \rightarrow \mathcal{T}_\Gamma(-D) \rightarrow \mathcal{T}_{\PP^n}|_\Gamma(-D) \rightarrow \mathcal{N}_{\Gamma/\PP^n}(-D)\rightarrow 0$$
and
$$0 \rightarrow \mathcal{O}_\Gamma(-D) \rightarrow \mathcal{O}_\Gamma(H-D)^{n+1} \rightarrow \mathcal{T}_{\PP^n}|_\Gamma(-D) \rightarrow 0,$$
combined with the fact that $H^1(\mathcal{O}_\Gamma(H-D))=0$ (by Serre duality, as $\mathrm{deg}(\mathcal{O}_\Gamma(H-D))>-2$), imply that $H^1(\mathcal{N}_{\Gamma/\PP^n}(-D))=0$. Thus, the dimension of $H_{\gamma,0,n}(p_1,\ldots,p_m)$ coincides with $\chi(\mathcal{N}_{\Gamma/\PP^n}(-D))$, and a straightforward computation using the two previous exact sequences yields
$$\chi(\mathcal{N}_{\Gamma/\PP^n}(-D))=(\gamma+1)(n+1)-4-m(n-1)=\dim(H_{\gamma,0,n})-m(n-1).$$

\end{proof}

\subsection*{Step 2}
As a consequence of the previous discussion and the fact that $m$ general points on a smooth rational curve of degree $\gamma$ lie in general position for $m\le \gamma +1$, we obtain:

\begin{lemma}
Let $\gamma_1,\gamma_2\ge 1$ and $m \le \min\{n+1,\gamma_2+1\}$ be integers.
 The incidence variety  
$$
I_{\gamma_1,\gamma_2,n}^m = \{(\Gamma_1,\Gamma_2) \mid   \Gamma_1 \cap \Gamma_2 = \{p_1,\ldots,p_m\}\} \subseteq H_{\gamma_1,0,n} \times H_{\gamma_2,0,n}
$$
of two smooth rational curves of degree $\gamma_1$ and $\gamma_2$, intersecting transversely at $m$ points in general position, is irreducible and unirational, of dimension 
\begin{align*}
\dim I_{\gamma_1,\gamma_2,n}^m  = \dim H_{\gamma_1,0,n} + m + \dim H_{\gamma_2,0,n} - m(n-1).
\end{align*}
\end{lemma}
\begin{proof}
By Lemma \ref{lemma:2rat} we have that the variety $H_{\gamma_2,0,n}(p_1,\ldots,p_m)$ is dominated by a variety $P(p_1,\ldots,p_m)$ isomorphic to a projective space, which parametrizes embeddings $f:\PP^1\hookrightarrow \PP^n$ of degree $\gamma_1$ whose image contains the points $p_i$.
We can construct a variety $\mathcal{P}$ over $H_{\gamma_1,0,n}\times (\PP^1)^m$ such that the fiber over a general point $(C,p_1,\ldots,p_m)$ with $p_1,\ldots,p_m\in C$ is the projective space $P(p_1,\ldots,p_m)$. $\mathcal{P}$ is generically a projective bundle over $H_{\gamma_1,0,n}\times (\PP^1)^m$ (at least on the open subset where the projection map is proper and flat). Hence $\mathcal{P}$ is irreducible and unirational. Moreover $\mathcal{P}$ dominates $I_{\gamma_1,\gamma_2,n}^m$, by sending an embedding $f$ in the fiber $P(p_1,\ldots,p_m)$ over $(C,p_1,\ldots,p_m)$ to the pair $(f(\PP^1),C)$. Thus $I_{\gamma_1,\gamma_2,n}^m$ is irreducible and unirational as well, since asking the intersection to be transversal is an open condition.
\end{proof}

\begin{lemma}
Let $n,d\ge 3$ and $\gamma \ge 1$ be integers.
 The incidence variety  
$$
I_{d,\gamma,n} = \{(E,\Gamma) \mid   E\cap \Gamma = \{pt\}\} \subseteq H_{d,1,n} \times H_{\gamma,0,n}
$$
of an elliptic and a smooth rational curve intersecting transversely at one point is irreducible and unirational, of dimension 
\begin{align*}
\dim I_{d,\gamma,n}  = \dim H_{d,1,n} + 1 + \dim H_{\gamma,0,n} - (n-1).
\end{align*}
\end{lemma}

\subsection*{Step 3}

Recall that a K3 surface of degree $2n-2$ with $3\le n \le 5$ is a complete intersection in $\PP^n$. More precisely, in the case $n=3$ the K3 surface is a quartic surface, while in the case $n=4$ (resp. $n=5$) the K3 surface is a complete intersection of a quadric and a cubic (resp. three quadrics). This is in fact the reason why we restrict to polarizations of degree $\le 5$.

In particular the choice of a K3 surface $X$ containing $E\cup \Gamma$ (or $\Gamma_1 \cup \Gamma_2$) is parametrized by an iterated Grassmannian $G$.  In the case $n=3$, $G=|\mathcal{I}_{E \cup \Gamma}(4)|$, in the case $n=4$, $G=\PP \mathcal{E}$ is a projective bundle over $|\mathcal{I}_{E \cup \Gamma}(2)|$, whose fiber over $q \in |\mathcal{I}_{E \cup \Gamma}(2)|$ is $\PP \mathcal{E}_q$ defined in the exact sequence
$$0 \rightarrow H^0(\mathcal{O}_{\PP^4}(1)) \stackrel{\cdot q}{\rightarrow} H^0(\mathcal{I}_{E \cup \Gamma}(3)) \rightarrow \mathcal{E}_q \rightarrow 0.$$
Finally, in the case $n=5$, our parameter space is $G=\mathrm{Gr}(3,H^0(E\cup \Gamma, \mathcal{I}_{E\cup \Gamma}(2)))$. All these parameter spaces are rational. We are going to discuss in detail Case 2, as Case 1 can be handled analogously.\newline

Let $k = d \gamma + d^2 - n + 1$. We consider the following incidence variety
$$
I := \{(X,E,\Gamma)\ |\  E\cup\Gamma\subseteq X, \ E \cap \Gamma =\{pt\}  \} \subseteq G\times I_{d,\gamma,n}.
$$

We denote by $I^{sm}\subseteq I$ the open subvariety of $I$ containing triples $(X,E,\Gamma)$ with $X$ smooth.

\begin{lemma} \label{lemma:nonempty}
The image of $\pi_1: I^{sm} \rightarrow \cM_{2k}$ sending a triple $(X,E,\Gamma)$ to the isomorphism class of the smooth K3 surface $X$ is open in $\cM_{2k}$. In particular, since $\cM_{2k}$ is irreducible, either $I^{sm} = \emptyset$ or $\pi_1$ is dominant.
\end{lemma}
\begin{proof}
Choose a basis $\{H,E,\Gamma\}$ of $U\oplus \langle -2k \rangle$ such that the intersection matrix of $\{H,E,\Gamma\}$ is as in \ref{genericLattice}, and let $X$ be a $U\oplus \langle -2k \rangle$-polarized K3 surface. Up to the action of the Weyl group, we can assume that $H \in \mathrm{NS}(X)$ is big and nef. By Saint-Donat's result \cite[Theorem 5.2]{SD}, $H$ is very ample if and only if there is no element $C\in \mathrm{NS}(X)$ with $C^2=-2$ and $HC=0$, and no element $D\in \mathrm{NS}(X)$ with $D^2=0$ and $HD=2$. Both conditions are closed in $\cM_{2k}$, and if $H$ is indeed very ample, it embeds $X$ in $\PP^n$ as a surface containing the desired pair of curves.
\end{proof}

In order to show the unirationality of $\cM_{2k}$, we will use the following proposition. We denote by $\pi_2: I \rightarrow I_{d,\gamma,n}$ the forgetful map that sends the triple $(X,E,\Gamma)$ to the pair $(E,\Gamma)$.

\begin{proposition} \label{prop:all}
Assume that there exists a pair $(E,\Gamma) \in I_{d,\gamma,n}$ such that the fiber $F:=\pi_2^{-1}(E,\Gamma)$ contains an element in $I^{sm}$. If the number $\dim{I_{d,\gamma,n}}+\dim{F}$ coincides with the \textit{expected dimension} of $I$
$$\dim{\cM_{2k}}+\dim{\mathrm{PGL}(n+1)}+\dim{|E|}+\dim{|\Gamma|}=18+\dim{\mathrm{PGL}(n+1)},$$
then $\pi_2$ is dominant. As a consequence $I$ is unirational, $\pi_1:I \dashrightarrow \cM_{2k}$ is a dominant rational map, and thus, $\cM_{2k}$ is unirational.
\end{proposition}
\begin{proof}
Since by assumption $I^{sm}\ne \emptyset$, then by Lemma \ref{lemma:nonempty} the morphism $\pi_1: I^{sm} \rightarrow \cM_{2k}$ is dominant and induces a dominant rational map $\pi_1 : I \dashrightarrow \cM_{2k}$. The map $\pi_1$ can be decomposed as the composition $q_2 \circ q_1$, where $q_1$ sends a triple $(X,E,\Gamma)$ to the K3 surface $X$, and $q_2$ sends a (smooth) K3 surface to its isomorphism class. The set of automorphisms of $\PP^n$ fixing a K3 surface $X\subseteq \PP^n$ is finite so that the dimension of the fiber of $q_2$ is equal to $\dim{\mathrm{PGL}(n+1)}$. Moreover, the fiber of $q_1$ is $1$-dimensional since the elliptic curve $E$ moves in a $1$-dimensional pencil and $\Gamma$ is rigid on $X$. Since $\pi_1$ is dominant, necessarily $\dim{I}$ is at least the expected dimension
$$\dim{\cM_{2k}}+\dim{\mathrm{PGL}(n+1)}+\dim{|E|}+\dim{|\Gamma|}=18+\dim{\mathrm{PGL}(n+1)}.$$
Now let $d \le \dim{F}$ be the minimal dimension of the fibers of $\pi_2$. By semicontinuity of the fiber dimension, there exists on open dense subset $U$ of $I_{d,\gamma,n}$ such that $\dim{\pi_2^{-1}(E',\Gamma')}=d$ for all $(E',\Gamma') \in U$. If $d < \dim{F}$, then
$$\dim{I}= \dim{I_{d,\gamma,n}} + d < \dim{I_{d,\gamma,n}}+\dim{F},$$
which is a contradiction, since by assumption the number $ \dim{I_{d,\gamma,n}}+\dim{F}$ coincides with the expected dimension of $I$. This implies that $\pi_2$ is dominant and that the general fiber of $\pi_2$ is isomorphic to $G$, and thus it gives to $I$ (generically) the structure of a projective bundle over $I_{d,\gamma,n}$ (since by cohomology and base change $\pi_2$ is a projective bundle over the open subset of $I_{d,\gamma,n}$ where $\pi_2$ is proper, flat and the fibers are isomorphic to $G$).
In particular $I$ is unirational, and therefore $\cM_{2k}$ is unirational as well.
\end{proof}

The previous proposition proves the existence of the diagram

\begin{equation}\label{incidenceVarietyK3}
 \xymatrix{
 & I \ar@{-->}[dl]_{\pi_1} \ar[dr]^{\pi_2} & \\ 
 \mathcal{M}_{2k} & & I_{d,\gamma,n}
 }
\end{equation}
whenever $I^{sm}\ne \emptyset$ and a certain equality of dimensions holds. As we already remarked, the general fiber of $\pi_2$ is isomorphic to the iterated Grassmannian $G$ defined above, of dimension
\begin{equation}\label{dimK3throughCurves}
\dim \pi_2^{-1}(E,\Gamma) =
\begin{cases}
 h^0(E\cup \Gamma, \mathcal{I}_{E\cup \Gamma}(4))-1 & \text{ for } n=3, \\
 h^0(E\cup \Gamma, \mathcal{I}_{E\cup \Gamma}(2)) - 1 + h^0(E\cup \Gamma, \mathcal{I}_{E\cup \Gamma}(3)) - 6 & \text{ for } n=4, \\
 3\cdot (h^0(E\cup \Gamma, \mathcal{I}_{E\cup \Gamma}(2))-3) & \text{ for } n = 5.
\end{cases}
\end{equation}
In the case $n=4$, we choose a quadric hypersurface through $E\cup \Gamma$ and then a cubic hypersurface through $E\cup \Gamma$ which is not a multiple of the chosen quadric.\\

We check with \textit{Macaulay2} that there exists a pair $(E,\Gamma) \in I_{d,\gamma,n}$ such that the fiber $F=\pi_2^{-1}(E,\Gamma)$ contains one element in $I^{sm}$ and that
$$\dim{I_{d,n,\gamma}}+\dim{F}-\dim{\mathrm{PGL}(n+1)}=18.$$
In order to compute the dimension of $I_{d,n,\gamma}$, we check with \textit{Macaulay2} that $(E,\Gamma)$ is a smooth point of $I_{d,n,\gamma}$.
Notice that the same strategy works analogously for Case 1, with the only difference that the previous number has to be $17$ instead of $18$. This follows from the fact that the two smooth rational curves are rigid on the K3 surface, and thus the fiber of $q_1$ in the proof of Proposition \ref{prop:all} is $0$-dimensional.
All the experimental data can be found in Tables \ref{table:ratrat}, \ref{table:ellrat}, \ref{table:nodal}. The interested reader may find the \textit{Macaulay2} code at \cite{M2support}.

We are also adapting the above strategy to prove unirationality for some quasi-polarized K3 surfaces. We construct K3 surfaces having a node (i.e. having a unique $A_1$-singularity) and containing either a smooth rational curve or a smooth elliptic curve (see Section \ref{nodalK3}). 

\subsection{Search for the lattices} \label{sub:search}

In this section we explain how we obtain an exhaustive list of projective models of elliptic K3 surfaces of Picard number $3$ in Case 1 and 2.

In Case 2, the K3 surface is automatically elliptic. In Case 1, we actually have to check that the lattice contains a copy of the hyperbolic plane. By \cite[Corollary 1.13.3]{Nik79}, the lattices $U\oplus \langle -2k\rangle$ are unique in their genus, so an even lattice $L$ of rank $3$, signature $(1,2)$ and determinant $2k$ contains a copy of $U$ if and only if the genus of $L$ coincides with the genus of $U\oplus \langle -2k \rangle$  (for the definition of the genus of a lattice, we refer to \cite{CS}). This amounts to computing the discriminant group of $L$; we do not report these straightforward computations here. 

In order to obtain an exhaustive, but finite, list of possible projective models for such K3 surfaces, we want to bound the degrees of the elliptic and smooth rational curves. This bound arises from the fact that curves of fixed genus and ``high'' degree do not lie on hypersurfaces of ``small'' degree. We explain this in detail for the case of two rational curves in $\PP^3$ meeting transversely at $m$ points in general position; the strategy in the other cases will be completely analogous.

Let $\Gamma_1,\Gamma_2$ be smooth rational curves in $\PP^3$ of degree $\gamma_1$ and $\gamma_2$, respectively, meeting at $m$ points. Assuming the maximal rank conjecture for the nodal union of two curves (see \cite{Lar17} for a proof in the smooth case), the short exact sequence
$$0 \rightarrow \mathcal{I}_{\Gamma_1 \cup \Gamma_2}(4) \rightarrow \cO_{\PP^3}(4) \rightarrow \cO_{\Gamma_1 \cup \Gamma_2}(4) \rightarrow 0$$
implies that $h^0(\mathcal{I}_{\Gamma_1 \cup \Gamma_2}(4))>0$ if and only if
\begin{equation} \label{eq:ineq}
    35=h^0(\cO_{\PP^3}(4)) > h^0(\cO_{\Gamma_1 \cup \Gamma_2}(4))=4(\gamma_1+\gamma_2)+2-m.      
\end{equation}
This allows us to obtain a (conjectural) bound on the degrees $\gamma_1$, $\gamma_2$ and on the number of intersection points $m$.
Indeed notice that the number $m$ of intersection points is bounded by the degrees $\gamma_1,\gamma_2$, since there are no smooth rational curves of degree $\gamma$ passing through more than $2 \gamma$ points of $\PP^n$ in general position (cf. the formula in Lemma \ref{lemma:2rat}). Therefore, the inequality \ref{eq:ineq} provides a bound for the degrees $\gamma_1,\gamma_2$ of the two smooth rational curves. More precisely, if $\gamma_1\ge \gamma_2$, then $m\le 2\gamma_2$ and therefore
$$35>4(\gamma_1+\gamma_2)+2-m\ge 4\gamma_1 + 2\gamma_1 + 2 \ge 4\gamma_1+4,$$
so $\gamma_2\le \gamma_1\le 7$.

Now, we can produce the list of all possible projective models of elliptic K3 surfaces of Picard number $3$ in Case 1 and 2: for every $\gamma_1\le \gamma_2 \le 7$ and every $m \le 2\gamma_1$ satisfying the inequality (\ref{eq:ineq}), we check whether the corresponding lattice contains a copy of the hyperbolic plane by looking at its genus.

The same search works analogously in the case of nodal K3 surfaces since we only have to bound the degree of the smooth rational curve or of the smooth elliptic curve.

\subsection{Primitivity} \label{prim}

The K3 surfaces that we construct in Tables \ref{table:ratrat}, \ref{table:ellrat}, \ref{table:nodal} are in fact $U\oplus \langle  -2k \rangle$-polarized K3 surfaces for suitable values of $k$.  In order to show this, we have to prove that the embeddings as in Equation \ref{genericLattice} are primitive. We will perform a case-by-case inspection, depending on the divisibility of $k$. Recall that, if $D \in U\oplus  \langle -2k \rangle$ is divisible by $r$ in $\mathrm{NS}(X)$, then $r^2 \mid k$.

\subsection*{Table \ref{table:ratrat}}
We consider the lattice embedding $\Lambda^1_{n,\gamma_1,\gamma_2} \hookrightarrow \mathrm{NS}(X)$ in Equation \ref{emb1}. Assume first $k\ne 50$. We observe that $k$ is either squarefree or a square, so we may assume that $k$ is a square (i.e. $k\in \{25,36,49,64,81\}$), otherwise the embedding is automatically primitive. It is easy to notice that, if $k=r^2$, then $\Gamma_1+\Gamma_2$ is a vector of square $0$ (since $m=2$ in all five cases) intersecting the three elements of the basis with multiplicity multiple of $r$. Therefore the embedding $\Lambda^1_{n,\gamma_1,\gamma_2} \hookrightarrow \mathrm{NS}(X)$ is primitive if and only if $\Gamma_1+\Gamma_2$ is not divisible in $\mathrm{NS}(X)$. Since $\Gamma_1+\Gamma_2$ represents the union of two rational curves on $X$ meeting at two points, it is a fiber of an elliptic fibration on $X$, and therefore it is not divisible in $\mathrm{NS}(X)$.

If instead $k=50$, the embedding is primitive if and only if the divisor $H-\Gamma_2$ is primitive in $\mathrm{NS}(X)$, since $H-\Gamma_2$ intersects the three elements of the basis with multiplicity $5$. If by contradiction $H-\Gamma_2$ were divisible by $5$, then $\frac{1}{5}(H-\Gamma_2)$ would have square $0$ and intersection $1$ with $H$; this is a contradiction since there are no curves of degree $1$ and arithmetic genus $1$.

\subsection*{Table \ref{table:ellrat}}
We consider the lattice embedding $\Lambda^2_{n,d,\gamma} \hookrightarrow \mathrm{NS}(X)$ in Equation \ref{emb2}. We are going to apply the following strategy for all the cases.\\
Let $D \in \Lambda_{n,d,\gamma}^2$ be a generator of $\langle E,\Gamma\rangle^\perp$. Since $\langle E,\Gamma\rangle$ is a copy of the hyperbolic plane, we have that $D^2=-2k$ and the basis $\{E,\Gamma,D\}$ gives an explicit isomorphism $\Lambda_{n,d,\gamma}^2 \cong U \oplus \langle -2k \rangle$. Therefore, the embedding $\Lambda_{n,d,\gamma}^2 \hookrightarrow \mathrm{NS}(X)$ is primitive if and only if $D$ is primitive in $\mathrm{NS}(X)$. Hence assume that $\frac{1}{2}D \in \mathrm{NS}(X)$ (thus $4 \mid k$); if $D$ is divisible by some other number $r$ the argument is analogous. 
A straightforward computation yields
$$D=H-(\gamma+2d)E-d\Gamma.$$
We distinguish some cases depending on the parity of $d,\gamma$.
 \\
$\bullet$  $d\equiv \gamma \equiv 0 \pmod{2}$: In this case $D$ is divisible by $2$ if and only if $H$ is divisible by $2$, but the hyperplane section is primitive in $\mathrm{NS}(X)$.
\\
$\bullet$ $d\equiv 1, \gamma \equiv 0 \pmod{2}$: $D$ is divisible by $2$ if and only if $H-\Gamma$ is divisible by $2$. But $\frac{1}{2}(H-\Gamma)$ would have square $0$ and intersection $2$ with $H$, and there are no curves of degree $2$ and arithmetic genus $1$.
\\
$\bullet$ $d\equiv 0, \gamma \equiv 1 \pmod{2}$: $D$ is divisible by $2$ if and only if $H-E$ is divisible by $2$. If $k=24$, $\frac{1}{2}(H-E)$ would be a curve of degree $2$ and arithmetic genus $1$. If $k=40$, $\frac{1}{2}(H-E)$ would have square $-2$, so it would be either effective or anti-effective. But $(H-E)H<0$, $(H-E)E>0$, and this is a contradiction since $H$ and $E$ are nef. Finally, if $k=68$, $\frac{1}{2}(H-E)$ would have square $-2$ and intersection $0$ with $H$; this is a contradiction since the K3 surfaces we are considering are generically smooth.
\\
$\bullet$ $d\equiv \gamma \equiv 1 \pmod{2}$: $D$ is divisible by $2$ if and only if $H-E-\Gamma$ is divisible by $2$. If $k \in \{16,28\}$, $\frac{1}{2}(H-E-\Gamma)$ would have square $-2$, but $(H-E-\Gamma)H<0$ and $(H-E-\Gamma)E>0$ leads to a contradiction as above. If instead $k=52$, $\frac{1}{2}(H-E-\Gamma)$ would be a $(-2)$-curve orthogonal to $H$ which is again a contradiction.

\subsection*{Table \ref{table:nodal}}
The reasoning is completely analogous to the previous case.

\section{Experimental data}\label{experimentalData}

\subsection{Case 1: Two smooth rational curves}

Let $\gamma_1,\gamma_2$ and $m$ be integers. 
Let $\Gamma_1,\Gamma_2 \subset \PP^n$ be two smooth rational curves of degree $\gamma_1$ and $\gamma_2$, respectively, intersecting transversely at $m$ points. If $X\subseteq \PP^n$ is a K3 surface containing the union $\Gamma_1 \cup \Gamma_2$, then there exists a lattice embedding
\begin{equation} \label{emb1}
\Lambda_{n,\gamma_1,\gamma_2}^1:=
\begin{pmatrix}
 2n-2 & \gamma_1 & \gamma_2 \\
 \gamma_1 & -2 & m \\ 
 \gamma_2 & m & -2
\end{pmatrix}
\hookrightarrow \text{NS}(X).
\end{equation}
For suitable choices, this lattice is isomorphic to $U \oplus \langle -2 k \rangle$ for some integer $k$. In Table \ref{table:ratrat} we specify this integer $k$ in every case and list the data obtained with our \textit{Macaulay2} program.
Note that there is more than one configuration of two smooth rational curves yielding a K3 surface in $\cM_{2k}$. We only list one possibility for each $k$ in Table \ref{table:ratrat}.

\subsection{Case 2: An elliptic and a smooth rational curve}

Let $d,n,\gamma$ be integers such that $3\le d\le 8$, $3\le n \le 5$ and  $1\le \gamma \le 7$. 
Let $E\subset \PP^n$ be an elliptic curve of degree $d$, and let $\Gamma\subset \PP^n$ be a smooth rational curve of degree $\gamma$ intersecting $E$ transversely at one point. We denote by $X\subseteq \PP^n$ a K3 surface with the following lattice embedding 
\begin{equation} \label{emb2}
\Lambda_{n,d,\gamma}^2 :=
\begin{pmatrix}
 2n-2 & d & \gamma \\
 d & 0 & 1 \\ 
 \gamma & 1 & -2
\end{pmatrix}
\cong U \oplus \langle -2( d^2 + d\gamma + -n +1) \rangle \hookrightarrow \mathrm{NS}(X).
\end{equation}
We set $k = d^2 +d \gamma -n+1$. In Table \ref{table:ellrat} we specify this integer $k$ in every case and list the data obtained with our \textit{Macaulay2} program. 

\begin{example}{$k=10$:} 
Let $X\subseteq \PP^3$ be a smooth  quartic surface containing a line and an elliptic curve of degree $3$ intersecting transversely at one point. Then $X$ is a K3 surface with the following primitive lattice embedding
$$
\begin{pmatrix}
 4 & 3 & 1 \\
 3 & 0 & 1 \\ 
 1 & 1 & -2
\end{pmatrix}
\cong U \oplus \langle -20 \rangle \hookrightarrow \mathrm{NS}(X).
$$

We recall the dimension count in this example:
\begin{align*}
& \dim(\text{Hilb}_{3\cdot t}(\PP^3)) + 1 + \dim(\text{Hilb}_{t+1}(\PP^3)) - 2  + (h^0(E\cup \Gamma, \mathcal{I}_{E\cup \Gamma}(4))-1) - \dim \text{PGL}(4)= \\
& = 12 + 1 + (4-2) + (19-1) - (4^2-1) = 18  = \dim \cM_{2k} + \dim |E| + \dim |\Gamma|.
\end{align*}
\end{example}

\begin{table}[H]
\centering
 \begin{tabular}{|c|c|c|c|c|c|c|c|}
  \hline
  $k$ & $n$ & $\gamma_1$ & $\gamma_2$ & $m$ & $\dim(\text{Hilb}_{\gamma_1\cdot t + 1}(\PP^n))$ & $\dim(\text{Hilb}_{\gamma_2\cdot t + 1}(\PP^n))$ & $\dim \pi_2^{-1}(\Gamma_1,\Gamma_2)$ \\ \hline \hline 
  $10$ & $3$ & $1$ & $1$ & $0$ & $4$ & $4$ & $24$ \\ \hline
  $13$ & $3$ & $2$ & $1$ & $0$ & $8$ & $4$ & $20$ \\ \hline
%  $13$ & $3$ & $2$ & $1$ & $1$ & $8$ & $4$ & $21$ \\ \hline
  $17$ & $4$ & $2$ & $1$ & $0$ & $11$ & $6$ & $24$ \\ \hline
  $19$ & $3$ & $3$ & $1$ & $1$ & $12$ & $4$ & $17$ \\ \hline
%  $19$ & $5$ & $2$ & $1$ & $2$ & $14$ & $8$ & $36$ \\ \hline
  $21$ & $4$ & $2$ & $2$ & $1$ & $11$ & $11$ & $21$ \\ \hline
%  $25$ & $3$ & $4$ & $1$ & $0$ & $16$ & $4$ & $12$ \\ \hline
%  $25$ & $4$ & $3$ & $2$ & $0$ & $16$ & $11$ & $14$ \\ \hline
  $25$ & $4$ & $3$ & $2$ & $2$ & $16$ & $11$ & $18$ \\ \hline
  $26$ & $3$ & $3$ & $3$ & $0$ & $12$ & $12$ & $8$ \\ \hline
%  $26$ & $5$ & $3$ & $1$ & $0$ & $20$ & $8$ & $24$ \\ \hline
  $29$ & $4$ & $4$ & $1$ & $0$ & $21$ & $6$ & $14$ \\ \hline
%  $29$ & $4$ & $4$ & $2$ & $3$ & $21$ & $11$ & $15$ \\ \hline
%  $29$ & $5$ & $3$ & $2$ & $0$ & $20$ & $14$ & $18$ \\ \hline
  $31$ & $5$ & $3$ & $2$ & $1$ & $20$ & $14$ & $21$ \\ \hline
  $34$ & $3$ & $5$ & $1$ & $0$ & $20$ & $4$ & $8$ \\ \hline
%  $34$ & $5$ & $3$ & $3$ & $0$ & $20$ & $20$ & $12$ \\ \hline
  $36$ & $5$ & $3$ & $3$ & $2$ & $20$ & $20$ & $18$ \\ \hline
  $37$ & $3$ & $5$ & $1$ & $1$ & $20$ & $4$ & $9$ \\ \hline
%  $37$ & $3$ & $5$ & $2$ & $0$ & $20$ & $8$ & $4$ \\ \hline
%  $37$ & $4$ & $4$ & $2$ & $1$ & $21$ & $11$ & $11$ \\ \hline
%  $37$ & $4$ & $4$ & $3$ & $4$ & $21$ & $16$ & $12$ \\ \hline
  $39$ & $5$ & $3$ & $3$ & $1$ & $20$ & $20$ & $15$ \\ \hline
  $41$ & $5$ & $4$ & $3$ & $0$ & $26$ & $20$ & $6$ \\ \hline
%  $41$ & $5$ & $4$ & $3$ & $3$ & $26$ & $20$ & $15$ \\ \hline
  $43$ & $3$ & $4$ & $3$ & $1$ & $16$ & $12$ & $5$ \\ \hline
%  $43$ & $5$ & $5$ & $1$ & $1$ & $32$ & $8$ & $15$ \\ \hline
  $46$ & $5$ & $5$ & $3$ & $4$ & $32$ & $20$ & $12$ \\ \hline
%  $49$ & $3$ & $6$ & $1$ & $1$ & $24$ & $4$ & $5$ \\ \hline
%  $49$ & $4$ & $5$ & $2$ & $2$ & $26$ & $11$ & $8$ \\ \hline
%  $49$ & $5$ & $4$ & $3$ & $1$ & $26$ & $20$ & $9$ \\ \hline
  $49$ & $5$ & $4$ & $3$ & $2$ & $26$ & $20$ & $12$ \\ \hline
  $50$ & $5$ & $5$ & $3$ & $0$ & $32$ & $20$ & $0$ \\ \hline
  $53$ & $5$ & $6$ & $1$ & $0$ & $38$ & $8$ & $6$ \\ \hline
  $59$ & $5$ & $5$ & $3$ & $3$ & $32$ & $20$ & $9$ \\ \hline
  $61$ & $5$ & $5$ & $3$ & $1$ & $32$ & $20$ & $3$ \\ \hline
  $64$ & $3$ & $5$ & $3$ & $2$ & $20$ & $12$ & $2$ \\ \hline
%  $64$ & $5$ & $5$ & $3$ & $2$ & $32$ & $20$ & $6$ \\ \hline
  $69$ & $5$ & $6$ & $3$ & $4$ & $38$ & $20$ & $6$ \\ \hline
  $73$ & $5$ & $5$ & $4$ & $4$ & $32$ & $26$ & $6$ \\ \hline
  $79$ & $5$ & $6$ & $3$ & $3$ & $38$ & $20$ & $3$ \\ \hline
  $81$ & $5$ & $6$ & $3$ & $2$ & $38$ & $20$ & $0$ \\ \hline
  $94$ & $5$ & $7$ & $3$ & $4$ & $44$ & $20$ & $0$ \\ \hline
  $97$ & $3$ & $5$ & $4$ & $4$ & $20$ & $16$ & $0$ \\ \hline
 \end{tabular}
 \caption{List of lattices in Case 1}
 \label{table:ratrat}
\end{table}

\begin{table}[H]
\centering
 \begin{tabular}{|c|c|c|c|c|c|c|}
  \hline
  $k$ & $n$ & $d$ & $\gamma$ & $\dim(\text{Hilb}_{d\cdot t}(\PP^n))$ & $\dim(\text{Hilb}_{\gamma\cdot t + 1}(\PP^n))$ & $\dim \pi_2^{-1}(E,\Gamma)$ \\ \hline \hline 
  $9$ & $4$ & $3$ & $1$ & $15$ & $6$ & $23$ \\ \hline
  $10$ & $3$ & $3$ & $1$ & $12$ & $4$ & $19$ \\ \hline
  $12$ & $4$ & $3$ & $2$ & $15$ & $11$ & $18$ \\ \hline
  $13$ & $3$ & $3$ & $2$ & $12$ & $8$ & $15$ \\ \hline
  $15$ & $4$ & $3$ & $3$ & $15$ & $16$ & $13$ \\ \hline
  $16$ & $3$ & $3$ & $3$ & $12$ & $12$ & $11$ \\ \hline
  $16$ & $5$ & $4$ & $1$ & $24$ & $8$ & $24$ \\ \hline
  $17$ & $4$ & $4$ & $1$ & $20$ & $6$ & $18$ \\ \hline
  $18$ & $3$ & $4$ & $1$ & $16$ & $4$ & $15$ \\ \hline
  $20$ & $5$ & $4$ & $2$ & $24$ & $14$ & $18$ \\ \hline
  $21$ & $4$ & $4$ & $2$ & $20$ & $11$ & $13$ \\ \hline
  $22$ & $3$ & $4$ & $2$ & $16$ & $8$ & $11$ \\ \hline
  $24$ & $5$ & $4$ & $3$ & $24$ & $20$ & $12$ \\ \hline
  $26$ & $3$ & $4$ & $3$ & $16$ & $12$ & $7$ \\ \hline
  $26$ & $5$ & $5$ & $1$ & $30$ & $8$ & $18$ \\ \hline
  $27$ & $4$ & $5$ & $1$ & $25$ & $6$ & $13$ \\ \hline
  $28$ & $3$ & $5$ & $1$ & $20$ & $4$ & $11$ \\ \hline
  $28$ & $5$ & $4$ & $4$ & $24$ & $26$ & $6$ \\ \hline
  $30$ & $3$ & $4$ & $4$ & $16$ & $16$ & $3$ \\ \hline
  $31$ & $5$ & $5$ & $2$ & $30$ & $14$ & $12$ \\ \hline
  $32$ & $4$ & $5$ & $2$ & $25$ & $11$ & $8$ \\ \hline
  $33$ & $3$ & $5$ & $2$ & $20$ & $8$ & $7$ \\ \hline
  $38$ & $3$ & $5$ & $3$ & $20$ & $12$ & $3$ \\ \hline
  $38$ & $5$ & $6$ & $1$ & $36$ & $8$ & $12$ \\ \hline
  $39$ & $4$ & $6$ & $1$ & $30$ & $6$ & $8$ \\ \hline
  $40$ & $3$ & $6$ & $1$ & $24$ & $4$ & $7$ \\ \hline
  $44$ & $5$ & $6$ & $2$ & $36$ & $14$ & $6$ \\ \hline
  $46$ & $3$ & $6$ & $2$ & $24$ & $8$ & $3$ \\ \hline
  $52$ & $5$ & $7$ & $1$ & $42$ & $8$ & $6$ \\ \hline
  $54$ & $3$ & $7$ & $1$ & $28$ & $4$ & $3$ \\ \hline
  $59$ & $5$ & $7$ & $2$ & $42$ & $14$ & $0$ \\ \hline
  $68$ & $5$ & $8$ & $1$ & $48$ & $8$ & $0$ \\ \hline
 \end{tabular}
 \caption{List of lattices in Case 2}
 \label{table:ellrat}
\end{table}

\section{Construction of nodal elliptic K3 surfaces} \label{nodalK3}

We adapt our above construction in order to deal with nodal K3 surfaces having one $A_1$-singularity. This allows us to complete the proof of Theorem \ref{thm:principal}. Indeed, for certain values of $k$, the constructions explained above only yield examples where the locus $I^{sm}$, corresponding to smooth K3 surfaces, is empty. Therefore, allowing in this section the K3 surfaces to have a singular point, we are able to prove the unirationality of $\mathcal{M}_{2k}$ for further values of $k$; notice that in all cases the general members of $\mathcal{M}_{2k}$ will be the minimal resolutions of the nodal K3 surfaces constructed here.

As in Step 1 of Construction \ref{construction}, we construct a curve of degree $d$ (that can either be a smooth rational or elliptic curve). In the second step we choose a point $p\in \PP^n$. The final step is to construct a K3 surface $X$ containing the given curve and having an $A_1$-singularity at $p$.  

We restrict our considerations to the case of an elliptic curve $E$ with a point $p$ on it (the other cases are treated similarly). The desingularization of $X$ contains the exceptional divisor $C_p \cong \PP^1$ over the $A_1$-singularity.  Then we have the following lattice embedding (which will also be primitive by Section \ref{prim})
\begin{equation*}
\begin{pmatrix}
 2n-2 & 0 & d \\
 0 & -2 & 1 \\ 
 d & 1 & 0
\end{pmatrix}
\cong \langle -2\cdot(d^2 - n +1) \rangle \oplus U \hookrightarrow \text{NS}(X).
\end{equation*}
given by the intersection matrix with respect to the basis $\langle \O_{\PP^n}(1)|_X, C_p, E \rangle$. 
We set $k'=d^2-n+1$. Hence, we have an incidence variety 
$$
I' := \{(X,E)\ |\ (E,p)\in H^1_{d,1,n} \text{ and } E\subset X \in \cM_{2k'} \} \subset \cM_{2k'}\times H^1_{d,1,n}
$$
and the natural projections, denoted by $\pi_1'$ and $\pi_2'$. The fiber $\pi_2'^{-1}(E,p)$ is unirational. Indeed, to obtain a K3 surface that is nodal in a point, one has to solve linear equations in the coefficients of the equations generating the K3 surface (as well as their derivatives). But the choice is unirational, and therefore the incidence variety $I'$ is unirational. The unirationality of $\cM_{2k'}$ follows by the construction of an example with the desired properties and a dimension count, as in Case 1 and 2 (see also Section \ref{sec:generalstrategy} for details). The following table lists our experimental data for such nodal K3 surfaces. We denote by $\dim \pi_2'^{-1}(E,p)$ the dimension of nodal K3 surfaces containing $E$ and having a node at $p\in E$.

\begin{table}[H]
\centering
 \begin{tabular}{|c|c|c|c|c|}
  \hline
  $k'$ & $n$ & $d$ & $\dim(\text{Hilb}_{d\cdot t}(\PP^n))$ &  $\dim \pi_2'^{-1}(E,p)$ \\ \hline \hline 
  $6$ & $4$ & $3$ & $15$ & $24$ \\ \hline
  $7$ & $3$ & $3$ & $12$ & $20$ \\ \hline
  $12$ & $5$ & $4$ & $24$ & $28$ \\ \hline
  $13$ & $4$ & $4$ & $20$ & $19$ \\ \hline
  $14$ & $3$ & $4$ & $16$ & $18$ \\ \hline
  $21$ & $5$ & $5$ & $30$ & $22$ \\ \hline
  $22$ & $4$ & $5$ & $25$ & $14$ \\ \hline
  $23$ & $3$ & $5$ & $20$ & $12$ \\ \hline
  $32$ & $5$ & $6$ & $36$ & $16$ \\ \hline
  $33$ & $4$ & $6$ & $30$ & $9$ \\ \hline
  $34$ & $3$ & $6$ & $24$ & $8$ \\ \hline
  $45$ & $5$ & $7$ & $42$ & $10$ \\ \hline
  $47$ & $3$ & $7$ & $28$ & $4$ \\ \hline
  $60$ & $5$ & $8$ & $48$ & $4$ \\ \hline
  $62$ & $3$ & $8$ & $32$ & $0$ \\ \hline
 \end{tabular}
  \caption{List of values for nodal K3 surfaces}
 \label{table:nodal}
\end{table} 

\begin{remark}
By studying nodal K3 surfaces containing a smooth rational curve with an $A_1$-singularity we can show the unirationality of $\cM_{2k}$ for the following values of $k$: 
$$
k\in \{13, 17, 21, 25, 31, 37, 41, 61\}.
$$
But we have already shown the unirationality of $\cM_{2k}$ for these values of $k$ before (see Table \ref{table:ratrat}). Therefore, we do not treat these cases in detail. 
\end{remark}

We end this section with three examples demonstrating how to choose a nodal K3 surface in $\PP^n$ for $n= 3,4$ and $5$ containing a given curve and having an $A_1$-singularity at a point. Furthermore, we recall the dimension count in these examples showing that the projection $\pi_1':I' \dashrightarrow \cM_{2k'}$ is dominant.

\begin{example}{Nodal quartic surfaces: $k'=7$.}
Given an elliptic curve $E$ together with a point $p$ in $\PP^3$, a nodal K3 surface containing $E$ and with a node at $p$ is a quartic generator of the ideal $\mathcal{I}_{Ep^2}:= \mathcal{I}_E\cap (\mathcal{I}_{p})^2$.

Let $X'\subset \PP^3$ be a nodal quartic surface containing an elliptic curve $E$ of degree $3$ with an $A_1$-singularity at a point of $E$. Then the desingularization $X$ of $X'$ is a smooth K3 surface with the following primitive lattice embedding
$$
\begin{pmatrix}
 4 & 0 & 3 \\
 0 & -2 & 1 \\ 
 3 & 1 & 0
\end{pmatrix}
\cong \langle -14 \rangle \oplus U \hookrightarrow \text{NS}(X).
$$
We recall the dimension count in this case:
\begin{align*}
& \dim(\text{Hilb}_{3\cdot t}(\PP^3)) + 1 + (h^0(\mathcal{I}_{Ep^2}(4))-1) - \dim \text{PGL}(4)= \\
& = 12 + 1 + (21-1) - (4^2-1) = 18  = \dim \cM_{2k'} + 1 .
\end{align*}
\end{example}

\begin{example}{Nodal complete intersections of a quadric and a cubic: $k' = 6$.}
Given an elliptic curve $E$ together with a point $p$, we get a nodal K3 surface by choosing two generators of degree $2$ and $3$ in the ideal $\mathcal{I}_E$ with the same tangent space at the point $p$. Therefore, we compute all quadric and cubic hypersurfaces containing $E$ and being tangent at $p$ to a $\PP^3$ that contains the tangent line of $E$ at $p$. The ideal of such hypersurfaces, denoted $\mathcal{I}_{E,\PP^3}$, is the intersection of $\mathcal{I}_E$ and $\mathcal{I}^2_{p}\cap \mathcal{I}_{\PP^3}$. 

Let $X'\subset \PP^4$ be a nodal K3 surface containing an elliptic curve $E$ of degree $3$ with an $A_1$-singularity at a point of $E$. Then the desingularization $X$ of $X'$ is a smooth K3 surface with the following primitive lattice embedding
$$
\begin{pmatrix}
 6 & 0 & 3 \\
 0 & -2 & 1 \\ 
 3 & 1 & 0
\end{pmatrix}
\cong \langle -12 \rangle \oplus U \hookrightarrow \text{NS}(X).
$$
We recall the dimension count in this case (cf. equation (\ref{dimK3throughCurves})):
\begin{align*}
& \dim(\text{Hilb}_{3\cdot t}(\PP^4)) + 1 +  \dim(\PP^3 \text{ containing the tangent line }  T_p(E)) \\
 & + (h^0(\mathcal{I}_{E,\PP^3}(2)) - 1 + h^0(\mathcal{I}_{E,\PP^3}(3)) -6) - \dim \text{PGL}(5)= \\
& = 15 + 1 + 2 + (7 + 24 - 7) - (5^2-1) = 18  = \dim \cM_{2k'} + 1 .
\end{align*} 

\end{example}

\begin{example}{Nodal complete intersections of three quadrics: $k'= 12$.}
Given an elliptic curve $E$ together with a point $p$, we obtain a nodal K3 surface by choosing a nodal quadric in the ideal $\mathcal{I}_{Ep^2}:= \mathcal{I}_E\cap (\mathcal{I}_{p})^2$ and two further quadrics in the ideal of $E$.

Let $X'\subset \PP^5$ be a nodal K3 surface containing an elliptic curve $E$ of degree $4$ with an $A_1$-singularity at a point of $E$. Then the desingularization $X$ of $X'$ is a smooth K3 surface with the following primitive lattice embedding
$$
\begin{pmatrix}
 8 & 0 & 4 \\
 0 & -2 & 1 \\ 
 4 & 1 & 0
\end{pmatrix}
\cong \langle -24 \rangle \oplus U \hookrightarrow \text{NS}(X).
$$
We recall the dimension count in this case (cf. equation (\ref{dimK3throughCurves}), noticing that the $-3$ in the first line comes from the projectivization of $H^0(\mathcal{I}_E(2)))$ and from the fact that the last two quadrics can be chosen up to multiples of the first quadric):
\begin{align*}
& \dim(\text{Hilb}_{4\cdot t}(\PP^5)) + 1 + (h^0(\mathcal{I}_{Ep^2}(2))-1) + (2\cdot(h^0(\mathcal{I}_E(2)) - 3)) - \dim \text{PGL}(6)= \\
& = 24 + 1 + (9-1) + (2\cdot(13 - 3)) - (6^2-1) = 18  = \dim \cM_{2k'} + 1 .
\end{align*}

\end{example}

\bibliographystyle{amsalpha}
\bibliography{bibliography}

\end{document}